\documentclass[12pt]{article}
\usepackage{amsmath,amsthm,amsfonts,graphicx,hyperref}
\usepackage[scale=0.75]{geometry}

\newtheorem{theorem}{Theorem}[section]

\newtheorem{conjecture}[theorem]{Conjecture}
\newtheorem{corollary}[theorem]{Corollary}
\newtheorem{lemma}[theorem]{Lemma}
\theoremstyle{definition}

\theoremstyle{remark}

\usepackage{dsfont}

\newcommand{\R}{\mathds R}
\newcommand{\F}{\mathcal F}
\DeclareMathOperator{\conv}{conv}

\begin{document}

\title{About an Erd\H{o}s--Gr\"{u}nbaum conjecture concerning piercing of non-bounded convex sets }
\author{Amanda Montejano \and Luis Montejano \and Edgardo Rold\'an-Pensado \and Pablo Sober\'on}
\maketitle
\abstract{In this paper, we study the number of compact sets needed in an infinite family of convex sets with a local intersection structure to imply a bound on its piercing number, answering a conjecture of
 Erd\H{o}s and Gr\"unbaum. Namely, if in an infinite family of convex sets in $\R^d$ we know that out of every $p$ there are $q$ which are intersecting, we determine if having some compact sets implies a bound on the number of points needed to intersect the whole family. We also study variations of this problem.}

\let\thefootnote\relax\footnotetext{The authors wish to acknowledgement the support of CONACYT under project 166306 and support of PAPITT-UNAM under projects IN112614 and IA102013.}
\let\thefootnote\relax
\section{Introduction}

The infinite version of the well known Helly theorem \cite{Hel1923} in the plane states the following: \textit{Given an infinite family of closed convex sets in the plane, one of which is bounded, if every three sets in the family have a common point, then the intersection of all them is non-empty}. Suggested by Erd\H{o}s in 1990, the following conjecture was first published in \cite{bs}: 
\textit{There is a constant $n$ such that, given any infinite family of closed convex sets in the plane, one of which is bounded, if among any four sets there are three with a point in common then there is a finite set $S$ consisting of $n$ points, such that every given set in the family contains at least one point from $S$.}

Eighteen years later, while reading the manuscript of the new edition of \cite{bs} Branko Gr\"{u}nbaum commented that this conjecture does not hold even for the line $\R$. He gave a construction that disproves the conjecture. Namely, define sets in $\R$ as follows: $F_0=\{0\}$, and $F_n=\{x\in \R\mid x\geq n\}$, for any positive integer $n$. Of course, all conditions of the conjecture are satisfied, while for any finite set $S$ of real numbers there is an integer $n$ that is greater than any number from $S$. Thus, by definition, $F_n$ does not contain any element from $S$. In the same year, 2008, Alexander Soifer asked Branko Gr\"{u}nbaum whether he could ``save" the conjecture. Consequently the following revised conjecture was published in \cite{s}.

\begin{conjecture}[Gr\"unbaum 2008]
\label{cjt:false2} There is an integer $n$ such that for any infinite family of closed convex sets in the plane, \emph{two} of which are bounded, if among any four sets there are three with a point in common, then there is a finite set $S$ (consisting of $n$ points), such that every set in the family contains at least one point from $S$.
\end{conjecture}

This conjecture was recently disproved by Tobias M\"uller \cite{muller2013counterexample}.  It should be noticed that a natural extension of M\"uller's work refutes the possibility of ``saving'' the conjecture by replacing the condition of \textit{two} bounded sets by any number. In this paper we study Erd\H{o}s' conjecture in the general setting of the $(p,q)$-problem of Hadwiger and Debrunner (see \cite{E, HD1957}), for infinite families of closed convex sets in $\R^d$. We say that a family $\F$ with at least $p$ closed convex sets in $\R^{d}$ satisfies the \emph{$(p,q)$-property} if among any $p$ sets in the family there are $q$ of them with a point in common. The \emph{piercing number}, $\pi (\F)$, is the minimum cardinality of a set $S\subset \R^d$, such that every set in the family contains at least one point from $S$. If there is no finite set intersecting the whole family, we simply say $\pi (\F) = \infty$. Hence the classical Helly theorem can be restated as follows:

\begin{theorem}[Helly, 1923]
\label{thm:helly} Let $\F$ be an infinite family of closed convex sets in $\R^{d}$, one of which is bounded. If $\F$ satisfies the $(d+1,d+1)$-property then $\pi (\F)=1$.
\end{theorem}

Hadwiger and Debrunner conjectured that the $(p,q)$-property should be enough to bound the piercing number of a finite family of convex sets, which was later confirmed by Alon and Kleitman \cite{AK}. The following theorem is now commonly known as the $(p,q)$-theorem.

\begin{theorem}[Alon, Kleitman 1992]\label{thm-alonkleitman}
Given positive integers $p \ge q \ge d+1$, there is a constant $c=c(p,q,d)$ such that every finite family $\F$ of closed convex sets in $\R^d$ with the $(p,q)$-property satisfies $\pi( \F) \le c$.
\end{theorem}

For the rest of the paper, we will denote by $\xi (p,q,d)$ the smallest possible value for the constant $c(p,q,d)$ of the theorem above. In this setting, it is natural to ask how many compact sets are necessary for the theorem above to hold for infinite families. Alon and Kleitman also proved an infinite version of Theorem \ref{thm-alonkleitman}, which we mention in Section \ref{section-disjuntos}.

Gr\"unbaum's example also shows that at least $p-q+1$ compact sets are necessary for the $(p,q)$-theorem to hold, just by taking $p-q$ copies of $F_0$ instead of only one. In this paper we characterise the triples $(p,q,d)$ such that $p-q+1$ compact sets in $\R^d$ are sufficient to imply the $(p,q)$-theorem for infinite families. Our main result is the following:

\begin{theorem}\label{thm:main} Let $p \geq q \geq d+1$ be positive integers.
\begin{itemize}
\item[i)] If $q \geq p-q +(d+1)$ and $\F$ is a family of closed convex sets in $\R{
R}^{d}$ containing at least $p-q+1$ bounded members and satisfying the $(p,q)$-property, then 
$$\pi(\F) \leq \xi(q-1,d,d-1)\xi(p,q,d) +p-q+1,$$ where $\xi (p,q,d)$ are the $(p,q)$-theorem bounds, and

\item[ii)] if $q < p-q +(d+1)$, then there is a family $\F$ of closed convex sets in $\R{
R}^{d}$, containing infinitely many bounded members, satisfying the $(p,q)$-property and such that $\pi(\F)=\infty$.
\end{itemize}
\end{theorem}

If we denote by $k$ the value of $p-q$, the theorem above can be restated as saying that for a family with the $(d+2k+1,d+k+1)$-property, having $k+1$ compact sets is enough to bound the piercing number, and for the $(d+2k,d+k)$-property no number of compact sets is enough. The proof of the positive part of Theorem \ref{thm:main} is in Section \ref{section-mainthm}.

In order to prove the negative part, we exhibit in Section \ref{section-counterexample} an infinite family $\F$ of closed convex sets in $\R^{d}$, infinitely many of  which are bounded, satisfying simultaneously the $(d+2k,d+k)$-property for every non-negative $k$ and for which the piercing number $\pi (\F)$ is infinite. In particular, for $d=2$, this construction is equivalent to M\"uller's counterexample \cite{muller2013counterexample}, disproving Conjecture~\ref{cjt:false2}.

Nevertheless, insisting on somehow \textquotedblleft saving" the spirit of Conjecture \ref{cjt:false2}, in Section \ref{section-disjuntos} we obtain positive results if some of the compact sets in $\F$ have a special separation structure.
Consider the following definition, we say that a family of convex sets in $\R^d$ is $m$-free if no $(m+1)$-tuple is intersecting and all its elements are compact. If $d \ge k$, one example of a $k$-free family of arbitrary size in $\R^d$ is to take a set of $(k-1)$-dimensional flats in general position intersected with a a large enough compact ball. Moreover, if these are all contained in a $k$-dimensional flat, then all $k$-tuples intersect while no $(k+1)$-tuples do.

\begin{theorem}
\label{thm:s2} Let $p\ge q \ge d+1$ be positive integers. If $\F$ is an infinite family of closed, convex sets in $\R^d$ with the $(p,q)$-property such that it contains a $(q-d)$-free family of size $p-d$, then $$\pi (\F) \le \xi(p,q,d) + p-q+1.$$
\end{theorem}

For example, with $(p,q,d)=(4,3,2)$, the theorem above says that in order to save Conjecture \ref{cjt:false2}, it is sufficient to have two \emph{disjoint} convex compact sets in the family.  This case was also noticed by M\"uller \cite{muller2013counterexample}, although he obtained a slightly stronger bound for the piercing number.  Namely, his work shows $\pi(\mathcal{F}) \le \xi(4,3,2)$.  If $d \ge k$, the construction for a $k$-free family mentioned above shows that Theorem \ref{thm:s2} implies bounds for the piercing number for special families with the $(d+2k,d+k)$-property (the breaking point of Theorem \ref{thm:main}).

One thing to note from the $(p,q)$-theorem is that the bounds obtained for $\xi(p,q,d)$ in \cite{AK} are astronomical. However, when they conjectured the $(p,q)$-theorem, Hadwiger and Debrunner showed that if $p$ and $q$ are large enough, then $\pi(\F) \le p-q+1$ (the best possible bound we could hope for). In the same spirit, in section \ref{section-erdosgallai} we show that if $p, q$ are large enough, then $p-q+1$ compact sets are enough to obtain the same bound on the piercing number for infinite families with the $(p,q)$-property. Namely,

\begin{theorem}
\label{thm:s1} Let $\F$ be an infinite family of closed convex sets in $\R^{d}$ containing at least $t+1$ bounded members.  Suppose $p-t \ge d+1$.  If $\F$ satisfies the $(p,p-t)$-property, for $p\geq \eta (d+1,t+1)$ then $\pi (\F)\leq t+1$, where $\ \eta (d+1,t+1)$ are the Erd\H{o}s--Gallai numbers.
\end{theorem}

The definition and history of the Erd\H{o}s--Gallai numbers are given in Section \ref{section-erdosgallai}. Here we just mention that $\eta(3,2)=6$ and recall the bound obtained by  Zsolt Tuza in  \cite{T}
\[
\eta(d+1,t+1) < \binom{d+t+1}{d} + \binom{d+t}{d},
\]

\noindent which implies that $\eta(3,3)\leq 15.$

Concerning convex sets in the plane (case $d=2$) the above results, the statement 78 of the book of Hugo Hadwiger and Hans Debrunner \cite{HD} in which they proved that   $\xi(p,q,1)=p-q+1$, plus the well known fact that  $\xi(4,3,2)\leq 13$ (see  \cite{KGT01, KT}), lead to the following:

\begin{itemize}
\item $(4,3)$-property $+$ infinitely many bounded sets $\not \Rightarrow$ $\pi (\F)< \infty$.
\item $(4,3)$-property $+$ two disjoint bounded sets $\Rightarrow$ $\pi (\F)\le 13 $, by \cite{muller2013counterexample}.
\item $(5,4)$-property $+$ two bounded sets $\Rightarrow$ $\pi (\F)\leq 28 $.
\item $(6,5)$-property $+$ two bounded sets $\Rightarrow$ $\pi (\F)\leq 2$.
\item $(15,13)$-property $+$ three bounded sets $\Rightarrow$ $\pi (\F)\leq 3$.
\end{itemize}

\section{Counterexample to Conjecture \ref{cjt:false2}} \label{section-counterexample}

Here we construct a counterexample to the following more general version of Conjecture \ref{cjt:false2}:\\
\emph{
Given a family $\F$ of closed convex sets in $\R^d$, infinitely many of which are bounded, if among any $d+2k$ sets there are $d+k$ with a point in common, then $\pi(\F)<\infty$.
}

This is false for every $d\ge1$ and $k\ge 0$ and the construction also proves the second part of Theorem \ref{thm:main}. This is summarised in the next theorem.
\begin{theorem}\label{thm:ctex}
There exist infinite families $\mathcal A$ and $\mathcal B$ of convex sets in $\R^{d}$ with the following properties:
\begin{itemize}
\item the piercing number $\pi(\mathcal A)=\infty$,
\item all members of $\mathcal B$ are bounded, and
\item the family $\F=\mathcal A\cup\mathcal B$ satisfies the $(d+2k,d+k)$-property for every $k\ge 0$.
\end{itemize}
\end{theorem}
The proof of this theorem follows immediately from lemmas \ref{lem:pi} and \ref{lem:aub} below.\\

In order to define the family $\mathcal A$, we first need the following auxiliary construction.  Denote by $e_1,\dots,e_d$ the standard basis of $\R^d$.
For a given number $0\le\alpha\le 1$, we define $S_\alpha$ as the $(d-1)$-dimensional simplex whose $k$-th vertex is $e_1+\dots+e_{k-1}+\alpha e_k$. In Figure \ref{fig:simp} there are two such simplices in dimensions $2$ and $3$.

\begin{figure}[ht]
\centering
\includegraphics[scale=0.75]{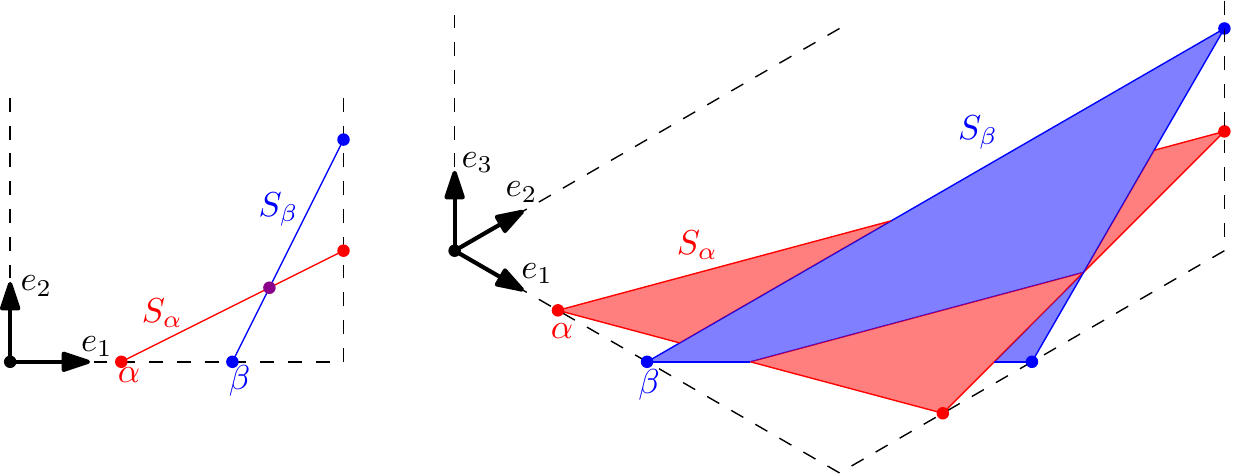}
\caption{Simplexes $S_\alpha$ and $S_\beta$ in $\R^2$ and $\R^3$.}
\label{fig:simp}
\end{figure}

\begin{lemma}\label{lem:simplex}
Given numbers $0 < \alpha_1 \le \dots \le \alpha_d<1$, the simplices $S_{\alpha_1},\dots,S_{\alpha_d}$ intersect.
\end{lemma}
\begin{proof}
The vertices of $S_\alpha$ arranged as rows in a matrix are
$$M_\alpha=
\begin{pmatrix}
\alpha & 0 & 0 & \dots & 0 & 0\\
1 & \alpha & 0 & \dots & 0 & 0\\
1 & 1 & \alpha & \dots & 0 & 0\\
\vdots & \vdots & \vdots & \ddots & \vdots & \vdots\\
1 & 1 & 1 & \dots & \alpha & 0\\
1 & 1 & 1 & \dots & 1 & \alpha
\end{pmatrix}.
$$

We use probability to explicitly construct coefficients for a convex combination of these vectors. The vector obtained will be a common point to all the simplexes. To simplify the exposition, we only construct the convex combination for the vertices of $S_{\alpha_d}$, the other cases are analogous.

Assume that $E_1,\dots,E_d$ are independent random events, each $E_i$ occurring with probability $\alpha_i$. For $k=0,\dots,d-1$, let $c_k$ be the probability that exactly $k$ of the first $d-1$ events occur. These are the coefficients for the convex combination, they are clearly non-negative and add up to $1$.

Now we compute the vector
$$\begin{pmatrix}
c_0 & c_1 & \dots & c_{d-1}
\end{pmatrix}M_{\alpha_d}.$$
The $i$-th coordinate of this vector is simply the probability that at least $i$ of the $d$ events occur. Since this is symmetric on $\alpha_1,\dots,\alpha_d$, we are done.
\end{proof}

We are ready to construct the family $\F$. To simplify notation, we construct the example in $\R^{d+1}$ and think of $\R^d\subset\R^{d+1}$ as the subspace with first coordinate equal to $0$. That is, $\R^d=\langle e_{2},\dots,e_{d+1}\rangle$.

Let the family $\mathcal B$ be any infinite family of bounded convex sets that contain the unit cube $[0,1]^d \subset \R^d$.

Let $\alpha_n=\frac 1n$ and consider the sets $S_{\alpha_n}\subset\R^d$. Then, define the rays $I_n=\{te_1\mid t\ge n\}$ and let $\mathcal A$ be the family of sets $A_n=\conv(S_{\alpha_n}\cup I_n)$ with $n\ge 2$. This is represented for $d=1$ in Figure \ref{fig:example}.

\begin{figure}[ht]
\centering
\includegraphics[scale=0.75]{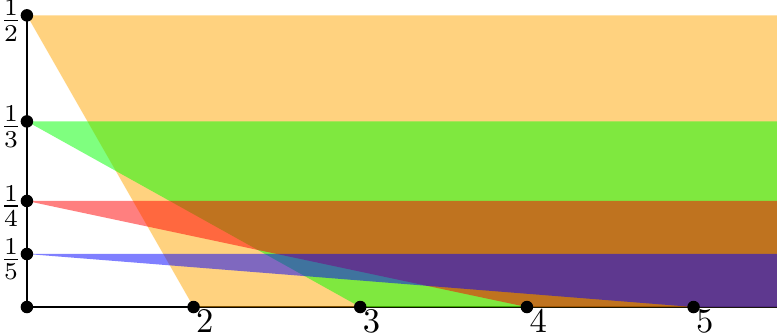}
\caption{The unbounded sets of $\F$ in $\R^2$.}
\label{fig:example}
\end{figure}

\begin{lemma}\label{lem:pi}
The piercing number $\pi(\mathcal A)=\infty$.
\end{lemma}
\begin{proof}
It is enough to show that any point $P$ in $\R^{d+1}$ is contained in a finite number of elements of $\mathcal A$. We show this by induction on $d$.

The case $d=1$ corresponds to Figure \ref{fig:example}. Let $P=(x,y)$. If $y\neq0$ then $P$ is not contained in any $S_n$ with $\frac 1n<\lvert y\rvert$. If $y=0$ then $P$ is not contained in any $S_n$ with $n>x$. In both cases $P$ is only in finitely many elements of $\mathcal A$.

Now assume that $d>1$ and let $P\in\R^{d+1}$. Note that the simplices $S_n$ tend to $S_0$, which is contained in $\langle e_2,\dots,e_d\rangle$.
Therefore, if the last coordinate of $P$ is not $0$ it can only be contained in finitely many elements of $\mathcal A$. If it is $0$, then by restricting to $\langle e_1,\dots,e_d\rangle$ we have the configuration corresponding to $d-1$. So by the induction hypothesis $P$ is contained only in finitely many elements of $\mathcal A$.
\end{proof}

\begin{lemma}\label{lem:aub}
The family $\F=\mathcal A\cup\mathcal B$ in $\R^{d+1}$ satisfies the $(d+1+2k,d+1+k)$-property for every $k\ge 0$.
\end{lemma}
\begin{proof}
Take $d+1+2k$ sets in $\F$ and let $\mathcal A'$ and $\mathcal B'$ be the families of the ones in $\mathcal A$ and in $\mathcal B$, respectively. Let $i=\lvert\mathcal A'\rvert$, so $d+1+2k-i=\lvert\mathcal B'\rvert$. We consider three cases:
\begin{enumerate}
\item If $i\le d$ then, by Lemma \ref{lem:simplex}, all the elements in $\mathcal A'\cup\mathcal B'$ intersect.
\item If $d+1\le i\le d+k$ then, by Lemma \ref{lem:simplex}, $d$ of the elements in $\mathcal A'$ and all the elements of $\mathcal B$ intersect. This gives $d+d+1+2k-i\ge d+1+k$ intersecting sets.
\item If $d+1+k\le i$ then all elements in $\mathcal A'$ intersect at some point in $\langle e_1\rangle$.
\end{enumerate}
In all cases there are at least $d+1+k$ intersecting sets.
\end{proof}

\section{Proof of Theorem \ref{thm:s2}}\label{section-disjuntos}

Let us start this section with the infinite version of the Hadwigwer-Debrunner $(p,q)$-theorem proved by Alon and Kleitman \cite{AK}.
 
\begin{theorem}[Alon, Kleitman 1992]\label{thm-AKinfinito}
Let $p$, $q$ and $d$ be positive integers with $p\geq q \geq d+1$. Then there exist a number $\xi(p,q,d)$ such that any infinite family $\F$ of closed convex sets in $\R^{d}$ satisfying the $(p,q)$-property can be partitioned into $\xi (p,q,d)$ subfamilies each satisfying the $(d+1,d+1)$-property.
\end{theorem}

\begin{proof}
If the family is countable, let $A_{1},\dots,A_{n}, \ldots$ be an enumeration of the members of $\F.$ Let $\F_{n}=\{A_{1},\dots,A_{n}\}$ and $V_{n}$ be the finite collection of all partitions of $\F_{n}$ into $\xi(p,q,d)$ subfamilies each satisfying the $(d+1,d+1)$-property. By \cite{AK}, for $n$ sufficiently big, $V_{n}$ is a non-empty finite set. Define a graph on $\bigcup_{1}^{\infty }V_{n}$ by inserting all edges $cc'$ such that $c\in V_{n}$ and $c'\in V_{n+1}$ is the restriction of $c$ to $\{A_{1},\dots,A_{n-1}\}$. By K\"{o}nig's infinite lemma, there is an infinite ray $c_{0}c_{1}\dots$ in this graph with $c_{i}\in V_{n_{0}+i}.$ Then $\bigcup_{1}^{\infty }c_{n}$\ is the desired partition. 

The proof for arbitrary infinite families follows directly from the compactness principle for partial selectors from finite sets. 
\end{proof}

We should point out that an infinite family of closed convex sets in $\R^d$ satisfying the $(d+1,d+1)$-property does not necessarily have a finite piercing number. In order to prove Theorem \ref{thm:s2} we will need the following lemma.

\begin{lemma}\label{lem:critical}
Let $p \ge q \ge d+1$ be positive integers and $\F$ be an infinite family of closed convex sets in $\R^d$ satisfying the $(d+1,d+1)$-property. Also, let $\mathcal{B}$ be a $(q-d)$-free family of size $p-d$. If $\F\cup \mathcal{B}$ satisfies the $(p,q)$-property then $\pi (\F)=1$.
\end{lemma}

\begin{proof}
Let $B$ be the convex hull of the sets in $\mathcal{B}$. By Helly's theorem, it will be enough to prove that $\F\cup \{B\}$ satisfies the $(d+1,d+1)$-property. Let $\F'$ be any subset of size $d$ of $\F.$ Consider the subfamily $\F'\cup\mathcal{B}$, of size $p$. Since $\F\cup \mathcal{B} $ satisfies the $(p,q)$-property, among these sets there is an intersecting $q$-tuple. Note that if this $q$-tuple does not contain $\F'$, then it has at least $q-d+1$ elements of $\mathcal{B}$, contradicting the fact that $\mathcal{B}$ is $(q-d)$-free. Since the $q$-tuple contains $\F'$ and at least one element of $\mathcal{B}$, we conclude that $\F' \cup \{ B\}$ is an intersecting family of size $d+1$, as we wanted.
\end{proof}

\begin{proof}[Proof of Theorem \ref{thm:s2}] Let $\F$ be an infinite family of closed convex sets in $\R^d$ satisfying the $(p,q)$-property with a $(q-d)$-free family $\mathcal{B}$ of size $p-d$.

Consider the family $\widehat{\F}=\F\setminus\mathcal{B}$. By definition $\widehat{\F}$ satisfies the $(p,q)$-property and, by Theorem \ref{thm-AKinfinito}, $\widehat{\F}=\F_{1}\cup \F_{2}\dots\cup \F_{K'}$ such that each $\F_{i}$ satisfies the $(d+1,d+1)$-property, where $K'=\xi (p,q,d)$. Now, by Lemma \ref{lem:critical} each subfamily $\F_{i}$ can be pierced with one point, hence $\pi (\mathcal{F)\leq }\xi (p,q,d)+p-d.$ However, if we take any $d$ elements of $\widehat{\F}$ and $\mathcal{B}$, the intersecting $q$-tuple must contain at least $q-d$ elements of $\mathcal{B}$. Thus to pierce $\mathcal{B}$ at most $p-q+1$ points are needed. Thus $\pi (\F) \le \xi(p,q,d) + p-q+1$.
\end{proof}

\section{Erd\H{o}s--Gallai Theory}\label{section-erdosgallai}

A vertex and an edge are said to \emph{cover} each other in a $\lambda$-hypergraph $G^{\lambda}$ if they are incident in $G^{\lambda}$. A \emph{vertex cover} in $G^{\lambda}$ is a set of vertices that covers all the edges of $G^{\lambda}$. The minimum cardinality of a vertex cover in $G^{\lambda}$ is called the \emph{vertex covering number} or \emph{transversal number} of $G^{\lambda}$ and is denoted by $\beta (G^{\lambda})$. For $\lambda \geq 2$, a $\lambda $-hypergraph with no isolated vertices $G^{\lambda}$ is called \emph{$k$-critical} if $\beta (G^{\lambda})=k$ and its transversal number decreases whenever an edge is deleted from $E(G^{\lambda})$. The study of $k$-critical $\lambda $-hypergraphs was initiated in 1961 with a paper of Erd\H{o}s and Gallai \cite{AG} in which they studied the maximum number of vertices $\eta (\lambda,k)$ that a $k$-critical $\lambda $-hypergraph $G^{\lambda}$ can have. Later, in 1989, Z.\ Tuza \cite{T} found sharp bounds for $\eta (\lambda,k)$. See \cite{AG1} for an excellent survey on the Erd\H{o}s--Gallai theory. Throughout the rest of the paper, the number $\eta(\lambda ,k)$ will be called the Erd\H{o}s--Gallai bound. Furthermore, Erd\H{o}s and Gallai's theorem can be restated as the following Helly type theorem for transversal numbers in hypergraphs.

\begin{theorem}[Erd\H{o}s, Gallai]\label{thm:eg}
Let $G^{\lambda}$ be a $\lambda$-hypergraph. Then $\beta (G^{\lambda})\leq k$ if and only if $\beta (H^{\lambda})\leq k$ for every $H^{\lambda}$ subgraph of $G^{\lambda}$ with $\lvert V(H^{\lambda})\rvert$ $\leq \eta (\lambda, k+1)$.
\end{theorem}

Clearly, by K\"{o}nig's infinite lemma or Tychonoff's theorem, we have that Theorem \ref{thm:eg} holds even if $G$ has an infinity number of vertices.

Let $\F$ be a family of closed convex sets in $\R^{d}$. As in \cite{OM}, we define a $(d+1)$-hypergraph $G_{\F}$ with vertex set $\F$ and $d+1$ convex sets of $\F$ define an edge of $G_{\F}$ if and only if their intersection is empty. We use the Erd\H{o}s--Gallai theorem to prove Theorem \ref{thm:s1}. See Theorem 3.1 of \cite{OM}.

\begin{proof}[Proof of Theorem \ref{thm:s1}] The condition that $\F$ satisfies the $(p,p-t)$-property, for $p\geq \eta (d+1,t+1)$ implies that $\beta (H)\leq t$, for every $H$ subgraph of $G_{\F}$ with $\lvert V(H)\rvert\leq \eta (d+1,t+1)$. The Erd\H{o}s--Gallai theorem implies that there is a transversal $\{A_{1},\dots,A_{t}\}$ $\subset \F$ to all edges of $G_{\F}$. So, by definition of $G_{\F}$, the family of closed convex sets in $\R^{d}$, $\F\setminus \{A_{1},\dots,A_{t}\}$ satisfies the $(d+1,d+1)$-property. The fact that $\F$ contains at least $t+1$ bounded members implies that at least one member of $\F\setminus \{A_{1},\dots,A_{t}\}$ is bounded, thus we can pierce $\F$ with $t+1$ points. 
\end{proof}

For example, using  that $\eta (n,2)\leq \lfloor (\frac{n}{2})^{2}\rfloor $, proved by  Erd\H{o}s and Gallai in \cite{AG}, (see also \cite{OM}),  plus Theorem \ref{thm:s1}, we have  that in the plane, a family of closed convex sets containing two bounded members and satisfying the $(6,5)$-property can be pierced with two points. Also, a family of closed convex sets in $3$-space containing two bounded members and satisfying the $(9,8)$-property can be pierced with two points.

\section{Proof of Theorem \ref{thm:main}}\label{section-mainthm}

In order to prove our main theorem we require some preliminary lemmas.

\begin{lemma}\label{lem:v}
Let $\F$ be an infinite family of closed sets in $\R^{d}$ satisfying the $(d+1,d+1)$-property and such that $\bigcap \{A \mid A\in \F\}=\emptyset$. Then, there is a unit vector $v$ such that for every $a\in A\in \F$, we have $\{a+tv\mid t\geq 0\}\subseteq A.$
\end{lemma}

\begin{proof}
Let $a\in A\subset \R^{d},$ where $A$ is an unbounded closed convex set. We define $C_{a}(A):=\{v\in S^{d-1}\mid \{a+tv\mid t\geq 0\}\subset A\}.$ By convexity and since $A$ is closed, this non-empty compact set is the same for every $a\in A,$ so we shall denoted it by $C(A).$ \ Note that if $B$ is a closed convex set and $A\subseteq B$, then $C(A)\subseteq C(B).$

Let $\F' \subset \F$ be a finite subset. 
By Helly's theorem $\bigcap\{A \mid A\in \F'\}$ is a non-empty unbounded closed convex set, otherwise $\bigcap \{A \mid A\in \F\} \neq \emptyset$. Therefore, $C(\bigcap\{A \mid A\in \F'\}) \subset \bigcap\{C(A ) \mid A\in \F'\}$ is non-empty. This implies that the family $\{ C(A) \mid A \in \F\}$, of compact subsets of $S^{d-1}$ has the finite intersection property. Therefore, $\bigcap\{C(A)\mid A\in F\} \neq \emptyset$ which conclude the proof of the lemma. 
\end{proof}

We have the following immediate corollary. 

\begin{corollary}\label{cor}
Let $\F$ be an infinite family of closed sets in $\R^{d}$ satisfying the $(d+1,d+1)$-property and such that $\bigcap \{A \mid A\in \F\}=\emptyset$. Let $B$ be a bounded closed convex subset of $\R^{d}$. Suppose without loss of generality that $v=(0,\dots,0,1)\in\R^{d}$ is the unit vector described in Lemma \ref{lem:v} for the family $\F$. Let $\F' \subset \F$, then $\bigcap \{A\cap B\mid A\in \F'\} \neq \emptyset$ if and only if $\bigcap \{\Pi(A\cap B)\mid A\in \F'\} \neq \emptyset$, where $\Pi :\R^{d}\to\R^{d-1}$ is the orthogonal projection. 
\end{corollary}

\begin{proof}
Let $x\in \bigcap \{\Pi(A\cap B)\mid A\in \F'\}$, then for every $A\in \F'$, there is $t_A \in \R$, such that $(x,t_A)\in A\cap B$. Let $t$ be a real number such that $t \ge t_A$, for every $A\in \F'$, and $(x,t) \in B$. Then, by Lemma \ref{lem:v}, $(x,t)\in \bigcap\{A\cap B\mid A\in \F'\}$.
\end{proof}

\begin{lemma}\label{lem:pixi}
Let $p \geq q \geq d+1$ be positive integers where $q \geq p-q +(d+1)$. Let $\F$ be a family of closed sets in $\R^{d}$ satisfying the $(d+1,d+1)$-property, and let $B_{1},\dots,B_{p-q+1}$ be compact convex sets such that $\F\cup \{B_{1},\dots,B_{p-q+1}\}$ satisfies the $(p,q)$-property. Then $$\pi(\F)\leq \xi(q-1,d,d-1),$$ where $\xi (p,q,d)$ are the $(p,q)$ theorem bounds.
\end{lemma}

\begin{proof}
We may assume that $\bigcap \{A\mid A\in \F\}=\emptyset $, otherwise there is nothing to prove. Suppose without loss of generality that $v=(0,\dots,0,1)\in\R^{d}$ is the unit vector described in Lemma \ref{lem:v} for the family $\F$. We denote by $B$ the convex hull of $\{B_{0},\dots,B_{p-q+1}\}$. Let us consider the following two families of compact convex sets: $B(\F)=\{A\cap B \mid A\in \F \}$, and $\Pi(\F)=\{\Pi(A\cap B) \mid A\in \F \}$ where $\Pi :\R^{d}\to\R^{d-1}$ is the orthogonal projection. Note that while $B(\F)$ is a family of sets in $\R^{d}$, $\Pi(\F)$ is a family of sets in $\R^{d-1}$. However, by Corollary \ref{cor} both families have the same pattern of intersection. 

We shall prove now that the family $B(\F)$ satisfies the $(q-1, d)$-property. 
For that purpose we use the fact that $\F\cup \{B_{1},\dots,B_{p-q+1}\}$ satisfies the $(p,q)$-property, and that $ p-q +d \leq q-1$. For a $q-1$-tuple $\{A_1\cap B,\dots,A_{q-1}\cap B \}$ of $B(\F)$ consider $\{A_1,\dots,A_{q-1},B_{1},\dots,B_{p-q+1}\}$, a $p$-tuple of $\F$. 
The fact that $\F\cup \{B_{1},\dots,B_{p-q+1}\}$ satisfies the $(p,q)$-property and that $ p-q +d \leq q-1$ implies that among the $q-1$ sets, $\{A_1\cap B,\dots,A_{q-1}\cap B \}$, there are $d$ of them with a point in common. Therefore, the family $B(\F)$ satisfies the $(q-1, d)$-property. 

By Corollary \ref{cor}, this implies that the family $\Pi(\F)$, of compact convex subsets of $\R^{d-1}$, satisfies the $(q-1, d)$-property. By the Alon Kleitman theorem, the family $\Pi(\F)$ can be pierced with $\xi(q-1,d,d-1)$ points, but again by Corollary \ref{cor}, the same is true for the family $B(\F)$ and of course for the family $\F$.
\end{proof}

\begin{proof}[Proof of Theorem 1.4, i)]
Let $B_{1},\dots,B_{p-q+1}$ be $p-q+1$ bounded members of $F$ and consider the family $\widehat{\F}=\F\setminus \{B_{1},\dots,B_{p-q+1}\}$. By definition $\widehat{\F}$ satisfies the $(p,q)$-property and, by Theorem \ref{thm-AKinfinito}, $\widehat{\F}=\F_{1}\cup \F_{2}\cup\dots\cup \F_{K}$ such that each $\F_{i}$ satisfies the $(d+1,d+1)$-property, where $K=\xi (p,q,d)$. Now, by Lemma \ref{lem:pixi} each subfamily $\F_{i}$ can be pierced with $\xi(q-1,d,d-1)$ points, hence $\pi(\F)\leq \xi(q-1,d,d-1)\xi(p,q,d) +p-q+1$.
\end{proof}

\begin{proof}[Proof of Theorem 1.4, ii)] Just check that the family $\F$ constructed in Theorem \ref{thm:ctex} satisfies the desired properties.
\end{proof}

\bibliographystyle{amsplain}
\bibliography{unbounded}

\noindent A. Montejano, \\
\textsc{
UMDI-Juriquilla Facultad de Ciencias, \\
Universidad Nacional Aut\'onoma de M\'exico
}\\[0.3cm]
\noindent L. Montejano, E. Rold\'an-Pensado \\
\textsc{
Unidad-Juriquilla, Quer\'etaro \\
Instituto de Matem\'aticas, \\
Universidad Nacional Aut\'onoma de M\'exico}\\[0.3cm]
\noindent P. Sober\'on \\
\textsc{
Mathematics Department, \\
University of Michigan, \\
Ann Arbor, MI 48109-1043
}\\[0.3cm]
\noindent \textit{E-mail addresses: }\texttt{amandamontejano@ciencias.unam.mx, luis@matem.unam.mx, e.roldan@im.unam.mx, psoberon@umich.edu}
\end{document}